\documentclass[12pt,reqno]{amsart} 

\usepackage[margin = 1.5in]{geometry} 
\usepackage{mathtools,amssymb,amsfonts,amsthm} 
\usepackage{color}
\usepackage[usenames,dvipsnames,svgnames,table]{xcolor}

\usepackage{url} 
\usepackage[linktoc=all,colorlinks=true,linkcolor=blue,urlcolor=Purple,citecolor=cyan]{hyperref} 
\usepackage[all]{hypcap} 
\usepackage{tikz} 
\allowdisplaybreaks

\newcommand{\atext}[1]{\text{\quad #1\quad}}

\newcommand{\bbn}{\mathbb{N}}
\newcommand{\bbz}{\mathbb{Z}}

\newcommand{\bmat}[1]{\begin{bmatrix}#1\end{bmatrix}}

\newcommand{\ord}{\operatorname{ord}}
\newcommand{\res}{\operatorname{Res}}

\newtheorem{thm}{Theorem}[section]
\newtheorem{prop}[thm]{Proposition}
\newtheorem{cor}[thm]{Corollary}
\newtheorem{lem}[thm]{Lemma}
\newtheorem*{prop*}{Proposition}
\newtheorem*{cor*}{Corollary}
\newtheorem*{lem*}{Lemma}
\newtheorem*{thm*}{Theorem}

\theoremstyle{definition}
\newtheorem{question}[thm]{Question}

\theoremstyle{remark}
\newtheorem{remark}[thm]{Remark}
\newtheorem*{remark*}{Remark}

\author[T. Alden Gassert]{T. Alden Gassert}
\address{%
Department of Mathematics and Computer Science, 
Hobart and William Smith Colleges, 
300 Pulteney Street, 
Geneva, NY 14456
}
\email{gassert@hws.edu}

\author[Michael T. Urbanski]{Michael T. Urbanski}
\address{%
Western New England University,
1215 Wilbraham Road,
Springfield, MA 01119
}
\email{michael.urbanski1@wne.edu}


\title{Index divisibility in the orbit of 0 for integral polynomials} 
\date{\today} 

\begin{document}

\begin{abstract}
Let $f(x) \in \bbz[x]$ and consider the index divisibility set $D = \{n \in \bbn : n \mid f^n(0)\}$.  We present a number of properties of $D$ in the case that $(f^n(0))_{n=1}^\infty$ is a rigid divisibility sequence, generalizing a number of results of Chen, Stange, and the first author.  We then study the polynomial $x^d + x^e + c \in \bbz[x]$, where $d > e \ge 2$ and determine all cases where this map has a finite index divisibility set.
\end{abstract}

\maketitle 

\section{Introduction}

Let $f(x) \in \bbz[x]$, and consider the orbit of $0$ under iteration by this function:
\begin{align*}
(f^n(0)) = (f^n(0))_{n=1}^\infty = (f(0), f^2(0), f^3(0), \ldots).
\end{align*}
Here $f^n(x)$ denotes the $n$-fold composition of $f$ with itself, and we also set $f^0(x) = x$.  If this sequence is unbounded, then $0$ is a \emph{wandering point}.  Otherwise $0$ is \emph{preperiodic}, and there exist integers $m \ge 1$ and $n \ge 0$ such that $f^{m+n}(0) = f^n(0)$.  If $n = 0$, then $0$ is \emph{periodic}, and the smallest positive integer $m$ for which $f^m(0) = 0$ is the \emph{exact period} of $0$.

In this dynamical setting, the orbit of $0$ is a \emph{divisibility sequence}.  That is, $f^m(0) \mid f^n(0)$ whenever $m \mid n$.  If $f(x)$ has no linear term (i.e. its linear coefficient is $0$) and $0$ is a wandering point, then $(f^n(0))$ is a \emph{superrigid divisibility sequence} \cite[Proposition 3.2]{r07}.  However, in this paper, we will only make use of the weaker condition that $(f^n(0))$ is a \emph{rigid divisibility sequence}.  A divisibility sequence $(a_n)$ is a rigid divisibility sequence if it satisfies the following properties.
\begin{enumerate}
\item If $v_p(a_n) \ge 1$, then $v_p(a_{nk}) = v_p(a_n)$ for all $k \ge 1$.
\item If $v_p(a_n) \ge 1$ and $v_p(a_m) \ge 1$, then $v_p(a_n) = v_p(a_m) = v_p(a_{\gcd(n,m)})$.
\end{enumerate}
Here, $v_p(n)$ denotes the $p$-adic valuation of $n$.


Given any sequence, it is natural to ask if the position of a value in the sequence reveals any information about the value itself.  In our case, we focus on the terms that are multiples of their indices.  These terms are captured by the \emph{index divisibility set}
\begin{align*}
D = D(f) = \{ n \in \bbn : n \mid f^n(0) \},
\end{align*}
where $\bbn$ is the set of positive integers.  

Historically, index divisibility has be studied in a variety of contexts.  For example, if $f(x) = a(x + a) - a$, then $f^n(0) = a^n - a$, and the question of index divisibility is analogous to the Fermat primality test.  Namely, if $n \nmid f^n(0)$, then $n$ is composite.  Otherwise if $n$ is relatively prime to $a$ and $n \mid f^n(0)$, then either $n$ is prime, or $n$ is a \emph{pseudoprime to base $a$}.  As another example, if one takes $f(x) = (x-1)^2 + 1$, then $f^n(a+1) = a^{2^n}+1$ is a generalized Fermat number, with $a = 2$ being the original case studied by Fermat.  The literature on index divisibility in Fibonacci and Lucas numbers (which are divisibility sequences) is extensive---see \cite{aj91,hvb74,j59,s17,s10,s93} as a sampling---and for general linear recurrences, see \cite{alps12}.  Silverman and Stange \cite{ss11} and Gottschlich \cite{g12} have studied this question for elliptic divisibility sequences, and Kim \cite{k17} considers the case where the $n$-th term in an elliptic divisibility sequence shares a fixed gcd with $n$.  In the dynamical setting, the index divisibility set for the polynomial $x^d + c \in \bbz[x]$ was analyzed by Chen, Stange, and the first author \cite{cgs17}.  

In \cite{cgs17}, the authors describe a graph whose vertex set is exactly the divisibility set for $f(x) = x^d + c$.  This \emph{index divisibility graph} $G$ is constructed iteratively as follows.  Start with $1$ as a vertex in $G$.  Then build out the rest of the graph by continuously looping through the vertices of $G$ and applying the rule: for each vertex $n$ in $G$ and each prime $p$, extend the graph by adding the vertex $np$ and the directed edge $(n,np)$ if either
\begin{enumerate}
\item $v_p(n) < v_p(f^n(0))$ (in which case $(n,np)$ is a \emph{type 1} edge), or
\item $v_p(n) = 0$ and $p \in D$ (and $(n,np)$ is a \emph{type 2} edge).
\end{enumerate}
We note that given any function $f$, such a graph may be constructed, and that leads us to the following generalization of \cite[Theorem 1.5]{cgs17}.

\begin{thm} \label{th:DivGraph}
Let $f(x) \in \bbz[x]$ and suppose $(f^n(0))$ is a rigid divisibility sequence. Let $D$ be its divisibility set and $G_V$ be the vertex set of its index divisibility graph. Then $G_V = D$. 
\end{thm}

A proof of this theorem is given in Section \ref{sec:DivisibilityGraph} along with generalizations of other statements from \cite{cgs17}. 

\begin{remark}
The index divisibility graph is a rooted directed graph with the vertex $1$ as its root.  We expect that the graph is infinite in most cases.  The edge types in the index divisibility graph are not mutually exclusive.  That is to say that there may be edges which are both type 1 and type 2.  The outdegree of each vertex depends on the number of primes in $D$ and hence may be finite or infinite.
\end{remark}

In Section \ref{sec:FiniteDivSet}, we study the trinomial $f(x) = x^d + x^e + c \in \bbz[x]$, where $d > e \ge 2$, and its divisibility set $D_{d,e,c}$.  In particular, we determine all cases where this set is finite.

\begin{thm} \label{th:DFinite}
The divisibility set $D_{d, e, c}$ is finite if and only if $c \in \{1, -1\}$.  Moreover, $D_{d,e,\pm 1} = \{1\}$.
\end{thm}

Given a sequence $(a_n)$, a prime $p$ is a primitive prime divisor of $a_n$ if $p \mid a_n$ and $p \nmid a_k$ for all $1 \le k < n$.  The terms in the sequence that do not have primitive prime divisors form the \emph{Zsigmondy set} of $(a_n)$:
\begin{align*}
Z((a_n)) = \{n \in \bbn : a_n \text{ has no primitive prime divisors}\}.
\end{align*}
In the construction of a divisibility graph, the main sources of edges emanating from a vertex $n$ are the \emph{primitive prime divisors} of $f^n(0)$.  Hence part of our strategy for proving Theorem \ref{th:DFinite} is to show that the divisibility set $D_{d,e,c}$ is contained in the Zsigmondy set of $(f^n(0))$ as this significantly restricts the potential for the divisibility set to be large.  We compute the Zsigmondy set of $f(x) = x^d + x^e + c$ explicitly in Proposition \ref{primitive_prime}.  Our proof is modeled after the argument of Doerksen and Haensch \cite{dh12}, who computed the Zsigmondy set for $x^d + c$.  It was already known to Rice that the Zsigmondy set for $x^d + x^e + c$ polynomial would be finite \cite[Theorem 1.2]{r07}, and since then the finiteness of Zsigmondy sets has been established in more general contexts \cite{gnt13,is09,s13}.

In the final section of the paper, we consider the primes in $D_{d,e,c}$.  For a prime $p$ to be in the divisibility set, it must be that $0$ is periodic modulo $p$, and that the period of $0$ is a divisor of $p$.  That is, either $0$ is fixed, in which case $p \mid c$, or the period of $0$ is $p$, in which case $f(x)$ is a cyclic permutation of $\bbz/p\bbz$.  Therefore the primes of most interest are those for which $f$ is a permutation polynomial with a prescribed cycle type.  For a survey of results on permutation polynomials, see Hou \cite{h15}, and see \cite{fg14,klmmss16,s08} for more on cycle structures of polynomials over finite fields.

In general it is difficult to guarantee the existence of specific primes in the index divisibility set.  For the map $x^d + x^e + c$, we find that if either $d$ or $e$ is even, then the only primes in $D_{d,e,c}$ are those dividing $c$ (Proposition \ref{only_primes}).  When both $d$ and $e$ are odd, it is not uncommon for $D_{d,e,c}$ to contain other primes.  In this case, we give conditions that would prevent primes from being in the divisibility set.

\section{Properties of the divisibility set} \label{sec:DivisibilityGraph}

In this section we identify properties of the index divisibility set for the polynomial $f(x) \in \bbz[x]$.  We then prove Theorem \ref{th:DivGraph}, showing that the divisibility graph defined in \cite{cgs17} yields the divisibility set for any $f(x) \in \bbz[x]$ where $(f^n(0))$ is a rigid divisibility sequence.  A number of these statements are more general versions of statements found in \cite{cgs17}, and for the most part, few changes are needed to adapt the arguments for our purposes.  We finish this section with a discussion on the divisibility graph in the case that $(f^n(0))$ is not a rigid divisibility sequence. 

\begin{prop} \label{pr:DivSetProp}
Suppose $f(x) \in \bbz[x]$, and let $D$ be its index divisibility set.
\begin{enumerate}
\item \label{d:div_c} If $n \mid f(0)$, then $n \in D$. 
\item \label{d:f_even} If $f(x)$ is an even function, then the only primes in $D$ are the primes dividing $f(0)$.
\item \label{d:valuation} If $n \in D$ and $v_p(n) < v_p(f^n(0))$, then $np \in D$.
\item \label{d:rel_prime} If $m,n \in D$ and $\gcd(m,n) = 1$, then $mn \in D$.
\item \label{d:smallest_prime} Suppose $m,n \in D$ and $m \mid n$.  Let $p$ be the smallest prime divisor of $n/m$. If $p \nmid m$, then $mp \in D$.

In particular, if $n \in D$ and $p$ is the smallest prime divisor of $n$, then $p \in D$.
\end{enumerate}
\end{prop}

\begin{proof}
\eqref{d:div_c} Suppose that $n \mid f^1(0)$. Since $(f^n(0))$ is a divisibility sequence, it follows that $f^1(0) \mid f^n(0)$, and thus $n \mid f^n(0)$.

\eqref{d:f_even} Suppose $f(x)$ is even and $p \in D$.  Necessarily, $0$ is periodic modulo $p$, and its period divides $p$.  If the period of $0$ is 1, then $f(0) = c \equiv 0 \pmod p$, and hence $p \mid c$.  

Otherwise, if the period of $0$ is $p$, then $f(f^{p-1}(0)) \equiv 0 \pmod p$, where $f^{p-1}(0) \not\equiv 0 \pmod p$.  However, since $f$ is even, it is also the case that $f(-f^{p-1}(0)) \equiv 0 \pmod p$.  Therefore $0$ has at least two preimages modulo $p$, and so the period of $0$ is strictly less than $p$ (a contradiction).

\eqref{d:valuation} Suppose that $n \in D$ and $v_p(n) < v_p(f^n(0))$.  Then $np \mid f^n(0)$. Since $(f^n(0))$ is a divisibility sequence, $f^n(0) \mid f^{np}(0)$, and hence $np \mid f^{np}(0)$. Therefore, $np \in D$.

\eqref{d:rel_prime} Suppose $m,n \in D$ and $\gcd(m, n) = 1$. Further assume $f^{mn}(0)$ is nonzero as otherwise the statement is trivial.  Since $(f^n(0))$ is a divisibility sequence, we have that $f^{m}(0) \mid f^{mn}(0)$ and $f^n(0) \mid f^{mn}(0)$. Therefore, $m \mid f^{mn}(0)$ and $n \mid f^{mn}(0)$.  Write $f^{mn}(0) = my$ and $f^{mn}(0) = nz$, where $y, z \in \bbn$.

Since $\gcd(m, n) = 1$, there exist $a,b \in \bbz$ such that $ma + nb = 1$.  Then
\begin{align*}
1 &= mn\left(\frac{a}{n} + \frac{b}{m}\right) \\
&= mn\left(\frac{az}{f^{mn}(0)} + \frac{by}{f^{mn}(0)} \right) \\ 
&= \frac{mn}{f^{mn}(0)}(az + by).
\end{align*}
Hence $mn(az + by) = f^{mn}(0)$, and thus $mn \in D$.

\eqref{d:smallest_prime} Suppose $m, n \in D$ and $m \mid n$.  Let $p$ be the smallest prime divisor of $n/m$, and suppose $p \nmid m$.  Since $p \mid n$ and $n \mid f^n(0)$, we have that $0$ is periodic modulo $p$.  Let $b$ denote the period of $0$ modulo $p$.  Note that $\gcd(b,n/m) \mid p$ since $\gcd(b,n/m)$ is a divisor of $n/m$ that is less than or equal to $p$.  Therefore, $b$ is either a divisor of $m$ or a divisor of $p$.  In the former case, $p \mid f^m(0)$.  Therefore $v_p(m) = 0 < v_p(f^m(0)$, and $mp \in D$ by part \eqref{d:valuation}.  In the latter case, it follows that $p \mid f^p(0)$, and hence $p \in D$.  Thus $mp \in D$ by part \eqref{d:rel_prime} since $\gcd(m,p) = 1$.
\end{proof}

We now prove Theorem \ref{th:DivGraph}.  For the benefit of the reader, we recall that the edges in the divisibility graph are all of the form $(n,np)$, where $p$ is prime.  The edge is type 1 if $v_p(n) < v_p(f^n(0))$, and it is type 2 if $p \in D$ and $v_p(n) = 0$.

\begin{proof}[Proof of Theorem \ref{th:DivGraph}]
We begin by showing that $G_V \subseteq D$. Certainly $1 \mid f^1(0)$, and so $1 \in D$. As the graph is constructed iteratively by adjoining edges of type 1 and type 2, it suffices to show that each vertex that is added to the graph in the construction is also in the index divisibility set. Hence we will examine each of these edge types and show that if $(n,np)$ is an edge in the graph and $n \in D$, then $np \in D$.

Suppose that $n \in D$ and $(n,np)$ is an edge in the divisibility graph. If $(n,np)$ is type 1, then $v_p(f^n(0)) > v_p(n)$. Hence $np \in D$ by Proposition \ref{pr:DivSetProp}.\eqref{d:valuation}. Otherwise $(n,np)$ is type 2, so $p \in D$ and $p \nmid n$. By Proposition \ref{pr:DivSetProp}.\eqref{d:rel_prime}, $np \in D$.

To show $D \subseteq G_V$, it suffices to show that for each $n \in D$, the divisibility graph contains a path from $1$ to $n$.  Let $n \in D$, write $n = \prod_{i=1}^k p_i^{\beta_i}$ for the prime factorization of $n$, and order the primes so that $p_1 < p_2 < \cdots < p_k$.  

Consider $m_j = \prod_{i=1}^{j-1} p_i^{\beta_i}$ for each $1 \le j \le k$, where we take $m_1 = 1$.  If $m_j \in D$, then following the proof of Proposition \ref{pr:DivSetProp}.\eqref{d:smallest_prime}, either $p_j \mid f^{m_j}(0)$ or $p_j \in D$.  If $m_j \in G_V$ and $p_j \mid f^{m_j}(0)$, then $(m_j,m_jp_j)$ is a type 1 edge.  Otherwise if $m_j \in G_V$ and $p_j \in D$, then $(m_j,m_jp_j)$ is a type 2 edge.

Moreover, if $m_j \in G_V$, then $(m_jp_j^t,m_jp_j^{t+1})$ is a type 1 edge for $1 \le t < \beta_j$ since 
\begin{align*}
v_p(f^{m_jp_j^t}(0)) = v_p(f^n(0)) \beta_j > t = v_p(m_jp_j^t).
\end{align*}
Thus if $m_j \in G_V$, then $m_{j+1} \in G_V$.  Since $m_1 \in G_V$, the divisibility graph contains a path from $1$ to $n$. 
\end{proof}

Consequently, we may expand our list of properties for the divisibility set in the case of rigid divisibility sequences. 

\begin{prop} \label{pr:RigidDivSetProp}
Suppose $f \in \bbz[x]$ and $(f^n(0))$ is a rigid divisibility sequence. Let $D$ be its index divisibility set.
\begin{enumerate}
\item \label{rd:smallest_prime} If $m,n \in D$, $m \mid n$, and $p$ is the smallest prime divisor of $n/m$, then $mp \in D$.
\item \label{rd:largest_prime} If $n \in D$ and $p$ is the largest prime divisor of $n$, then $n/p \in D$.
\end{enumerate}
\end{prop}

\begin{proof}
Part \eqref{rd:smallest_prime} differs from Proposition \ref{pr:DivSetProp}.\eqref{d:smallest_prime} in that we allow for $p$ to divide $m$.  If $p \mid m$, then by rigid divisibility, $v_p(f^m(0)) = v_p(f^n(0)) \ge v_p(n) > v_p(m)$.  Thus $mp \in D$ by Proposition \ref{pr:DivSetProp}.\eqref{d:valuation}.  We note that if $(f^n(0))$ is only a divisibility sequence, then it may be that $v_p(f^m(0)) = v_p(f^{mp}(0))$, in which case $mp \nmid f^{mp}(0)$.

Part \eqref{rd:largest_prime} comes directly from the construction of the path from $1$ to $n$ in the proof of Theorem \ref{th:DivGraph}. Namely, if $p$ is the largest prime divisor of $n$, the edge $(n/p,n)$ is the last edge in the path.
\end{proof}

We also note that one may recover a divisibility graph directly from a divisibility set.

\begin{prop} \label{pr:DtoG}
If $D$ is the divisibility set for a rigid divisibility sequence, then associated divisibility graph has vertex set $G_V = D$ and edge set 
\begin{align*}
G_E = \{(m,n) : m,n \in D \text{ and } n/m \text{ is prime}\}.
\end{align*}
\end{prop}

\begin{proof}
Certainly $G_V = D$ and $G_E \subseteq \{(m,n) : m,n \in D \text{ and } n/m \text{ is prime}\}$.  For the reverse inclusion, the argument is identical to the final paragraphs in the proof of Theorem \ref{th:DivGraph}.  Briefly, suppose $m,n \in D$, and let $p = n/m$ be prime.  If $p \nmid m$, then $(m,mp)$ is an edge of type 1 or type 2 depending on whether the period of $0$ modulo $p$ divides $m$ or divides $p$.  If $p \mid m$, then $(m,mp)$ is type 1.
\end{proof}

To conclude this section, we consider possibilities for the divisibility graph of $f(x) \in \bbz[x]$ in the case that $(f^n(0))$ is not a rigid divisibility sequence.  We point out that at a glance, the definition of the divisibility graph presented above seems inadequate for divisibility sequences.  For instance, if one uses the definition above to construct the divisibility graph for the sequence of natural numbers $(1, 2, 3, \ldots)$, then one quickly finds that there are no type 1 edges and that the graph contains infinitely many components.  If one uses Proposition \ref{pr:DtoG} to define the divisibility graph, then the graph for the sequence of natural numbers will be connected.  However, this too has its shortcomings.  For one, what independence the graph had from $D$, it now loses.  Nor does the statement in Proposition \ref{pr:DtoG} guarantee that the graph is rooted, much less connected.  That is, even if the graph is comprised of a single component, it may not be possible to reach every vertex in the graph from 1 via a sequence of directed edges.

Experimentally, however, the current definition of the divisibility graph appears to be robust.  As a small survey, we computed 
\begin{align*}
\{n \in \bbn : n \mid f^n(0) \text{ and } n \le 5000\}
\end{align*}
for the maps $x^3 + x + c$ and $x^4 + x + c$, where $c \in \{1, 2, 3, \ldots, 100\}$.  We then constructed their divisibility graphs and verified that every edge in these graphs were either type 1 or type 2.  This begs the following question.

\begin{question}
Does Theorem \ref{th:DivGraph} apply to all $f(x) \in \bbz[x]$? Otherwise, is there an $f(x) \in \bbz[x]$ whose index divisibility set contains values $n$ and $np$, but $(n,np)$ is neither type 1 nor type 2?
\end{question}

Recalling Rice \cite[Proposition 3.2]{r07}, if $f \colon \bbz \to \bbz$ and $(f^n(0))$ is not a rigid divisibility sequence, then the coefficient of its linear term is nonzero.  Writing $f(x) = x^2g(x) + bx + c$ where $g(x) \in \bbz[x]$, it is straightforward to verify that 
\begin{align*}
f^n(0) = c^2h(c) + c\sum_{i=0}^{n-1}b^i.
\end{align*}
for some $h(x) \in \bbz[x]$.  Note that for all primes $p$,
\begin{align*}
\sum_{i=0}^{p-1} b^i \equiv 
\begin{cases*}
0 & if $b \equiv 1 \pmod p$\\
1 & otherwise.
\end{cases*}
\end{align*}
Thus for the primes dividing $c$, either $b \not\equiv 1 \pmod p$ and $v_p(f^n(0)) = v_p(c)$ for all $n \in \bbn$, or $b \equiv 1 \pmod p$ and $v_p(f^{np}(0)) > v_p(np)$ for all $n \in \bbn$.  Therefore all the edges in the divisibility graph that result from primes dividing $c$ are type 1.  The question is still open for primes that do not divide $c$.  

\section{The polynomial \texorpdfstring{$x^d + x^e + c$}{}} \label{sec:FiniteDivSet}

In this section, we restrict our attention to the polynomial $f(x) = x^d + x^e + c \in \bbz[x]$, where $d > e \ge 2$.  We begin with a pair of propositions regarding primes in the index divisibility set for $f(x)$.  We then turn to the topic of primitive primes divisors, and in Proposition \ref{primitive_prime}, show that the Zsigmondy set for $f(x)$ is a subset of $\{1\}$.  Following that, we determine all cases where the index divisibility set of $f(x)$ is finite, proving Theorem \ref{th:DFinite}.

Throughout this section, we let $D_{d,e,c}$ denote the index divisibility set for $f(x) = x^d + x^e + c$, and for convenience, we set $O_{d,e,c} = (f^n(0))$ and $O^+_{d,e,c} = (|f^n(0)|)$.

\begin{prop} \label{only_primes}
If $d$ or $e$ is even and $p \in D_{d,e,c}$, then $p \mid c$. 
\end{prop}

\begin{proof}
If $d$ and $e$ are both even, then $f(x)$ is an even function and Proposition \ref{pr:DivSetProp}.\eqref{d:f_even} applies.

In the case that exactly one of $d$ or $e$ is even, we have $f(-1) = f(0) = c$.  Therefore, $c$ has two preimages in $\mathbb{Z}/p\mathbb{Z}$ for every prime $p$. Hence, $0$ can not have period $p$ modulo $p$.  Thus $p \in D_{d,e,c}$ only if $0$ is fixed modulo $p$, i.e. $p \mid c$.
\end{proof}

\begin{cor} \label{all_type_1}
If $d$ or $e$ is even, then every edge in the index divisibility graph associated to $f(x) = x^d + x^e + c$ is type 1.
\end{cor}

\begin{proof} 
Suppose $(n, np)$ is a type 2 edge in the index divisibility graph for $f(x)$.  Then $p \in D_{d,e,c}$ and $v_p(n) = 0$. If $d$ or $e$ is even, then by Proposition \ref{only_primes}, $p \mid c$. Since $O_{d,e,c}$ is a divisibility sequence, $p \mid f^n(0)$. Therefore $v_p(f^n(0)) > v_p(n)$, and we have that $(n, np)$ is a type 1 edge.   
\end{proof}

\begin{prop} \label{prop 9}
If $p \in D_{d,e,c}$, then $p \in D_{d+k_1(p-1), e + k_2(p - 1), c}$ for all $k_1,k_2 \in \bbz$, where $d + k_1(p - 1) \geq 3$, and $e + k_2(p - 1) \geq 2$.
\end{prop}

\begin{proof}
Let $p \in D_{d,e,c}$ and consider the polynomial $g(x) = x^{d + k_1(p - 1)} + x^{e + k_2(p - 1)} + c$. Then
\begin{align*}
g(x) & = x^{d + k_1(p - 1)} + x^{e + k_2(p - 1)} + c \\
     & = x^d \cdot x^{k_1(p - 1)} + x^ex^{k_2(p - 1)} + c \\
     & \equiv x^d + x^e + c \pmod p.
\end{align*}
So $g^p(0) \equiv f^p(0) \equiv 0 \pmod p$. Thus, $p \in D_{d+k_1(p-1), e + k_2(p - 1), c}$.
\end{proof}

We now give several technical lemmas, which will be useful for determining the Zsigmondy set of $O_{d,e,c}$.

\begin{lem} \label{lem:increasing}
Let $f$ represent a polynomial of the type $f(x) = x^d + x^e + c$ such that $d > e \ge 2$ and $|c| > 1$. Then, $O^+_{d, e, c}$ is a strictly increasing sequence.
\end{lem}

\begin{proof}
Suppose $|c| > 1$ and $d > e \ge 2$. We proceed by induction. For the base case, we have
\begin{align*}
|f^2(0)| & = |c^d + c^e + c| \\
         & = |c| \cdot |c^{d - 1} + c^{e - 1} + 1| \\
         & > |c| \cdot |c^2 + c + 1| \\
         & > |c| \cdot (|c^2 + 1| - |c|) \\
         & > |c| = |f^1(0)|.
\end{align*}
Now assume $| f^n(0) | > | c |$ for some $n$. We have
\begin{align*}
|f^{n + 1}(0)| & = |(f^n(0))^d + (f^n(0))^e + c| \\
                & \geq |f^n(0)((f^n(0))^{d - 1} + (f^n(0))^{e - 1})| - |c| \\
			    & = |f^n(0)| \cdot |(f^n(0))^{d - 1} + (f^n(0))^{e - 1}| - |c| \\
                & \geq |f^n(0)| \cdot ||f^n(0)|^{d - 1} - |f^n(0)|^{e - 1}| - |c| \\
                & \geq |f^n(0)| \cdot ||c|^{d - 1} - |c|^{e - 1}| - |c| \\
                & \geq |f^n(0)| \cdot |c| - |c| \\
                & > |f^n(0)|.
\end{align*}
\end{proof}

\begin{lem} \label{lem 14}
If $f(x) = x^d + x^e + c$ where $d > e \ge2$, then either
\begin{enumerate}
\item 0 is a wandering point and $O^+_{d, e, c}$ is a increasing sequence, or
\item 0 is a preperiodic point, which occurs exactly when
	\begin{enumerate}
    \item $c = 0$, or
    \item $c = -1$ and either $d$ or $e$ is even.
    \end{enumerate}
\end{enumerate}
\end{lem}

\begin{proof}
The case where $|c| > 1$ is precisely Lemma \ref{lem:increasing}. 
In the case that $c = 1$, simple induction can be used to show that $O_{d,e,1}$ is an increasing sequence, and a similar argument applies in the case where $d$ and $e$ are both odd and $c = -1$. In fact, $O^+_{d,e,-1} = O_{d,e,1}$.

In the case that $c = 0$, it can easily be seen that $f^1(0) = 0$.
Otherwise, let $c = -1$. If exactly one of $d$ and $e$ is even, then $f^2(0) = -1 = f^1(0)$. In the case when $d$ and $e$ are both even we find that $f^3(0) = 1 = f^2(0)$. 
\end{proof}

Recall that if $a_n$ is a term in the sequence $(a_n)$, the primitive prime divisors of $a_n$ are the primes that do not divide $a_i$ for $1 \le i < n$.  Thus we may distinguish between the primitive and non-primitive primes of $a_n$ and write $a_n = P_nN_n$, where $P_n$ is the primitive part of $a_n$ and $N_n$ is the non-primitive part of $a_n$.  That is, $P_n$ is a product of powers of primitive primes of $a_n$, and $N_n$ is a product of powers of non-primitive primes.

\begin{lem} \label{lem_Doerksen/Haensch}
If $(a_n)$ is a rigid divisibility sequence, then
\begin{align*}
N_n = {\displaystyle \prod_{d|n, d \not = n} P_d}.
\end{align*}
\end{lem}

\begin{proof}
See \cite[Lemma 6]{dh12}.
\end{proof}

The following result determines the Zsigmondy set for $f(x)$.

\begin{prop} \label{primitive_prime}
Let $f(x) = x^d + x^e + c$, where $d > e \ge 2$. If 0 is a wandering point, then
\begin{enumerate}
\item if $c = \pm 1$, $f^n(0)$ has a primitive prime divisor for all $n \geq 2$, and
\item if $c \ne \pm 1$, $f^n(0)$ has a primitive prime divisor for all $n \geq 1$.
\end{enumerate}
\end{prop}

\begin{proof}
Assume $0$ is a wandering point. Based off of Lemma \ref{lem 14}, we can eliminate cases where $c = -1$ and where $c=0$ when either $d$ or $e$ is even. In all other cases, 0 is a wandering point.

Note that if $c = \pm 1$, then $f(0) = \pm 1$, in which case $f(0)$ does not have a primitive prime divisor. If $c \ne \pm 1$, then $f^1(0) = c$ has at least one primitive prime factor, namely any prime factor of $c$.

For $n = 2$ and $|c| \ge 1$, we have that
\begin{align*}
f^2(0) = c(c^{d - 1} + c^{e - 1} + 1).
\end{align*}
From Lemma \ref{lem 14}, the sequence $O^+_{d, e, c}$ increasing, hence $|c^{d-1}+c^{e-1}+1|>1$. Therefore $f^2(0)$ has primitive prime divisors, namely the divisors of $c^{d-1}+c^{e-1}+1$. 

Now we proceed to show that $f^n(0)$ has a primitive prime divisor for all $n \ge 3$. First we derive a result that will be helpful later. Assume $n \ge 3$, then 
\begin{align*}
|f^n(0)| & = |(f^{n - 1}(0))^d + (f^{n - 1}(0))^e + c| \\
		 & \geq |(f^{n - 1}(0))^d + (f^{n - 1}(0))^e| - |c| \\
         & \geq |f^{n - 1}(0)|^d - |f^{n - 1}(0)|^e - |c| \\
         & \geq |f^{n - 1}(0)|^3 - |f^{n - 1}(0)|^2 - |c| \\
         & > |f^{n - 1}(0)|^3 - |f^{n - 1}(0)|^2 - |f^{n - 1}(0)| \\
         & > |f^{n - 1}(0)|^3 - |f^{n - 1}(0)|^2 - |f^{n - 1}(0)|^2 + |f^{n - 1}(0)| \\
         &= |f^{n - 1}(0)|^3 - 2|f^{n - 1}(0)|^2 + |f^{n - 1}(0)|.
\end{align*}
Factoring out $|f^{n - 1}(0)|$ gives
\begin{align} \label{eq:PrimPrime}
|f^n(0)| > |f^{n - 1}(0)|(|f^{n - 1}(0)|^2 - 2|f^{n - 1}(0)| + 1).
\end{align}

Next, we show that
\begin{align*}
\prod_{k=1}^{n - 1} |f^k(0)| < |f^n(0)|.
\end{align*}
We proceed by induction. The base case, $n = 3$, may be checked readily. Now assume that $\prod_{k=1}^{n - 2} |f^k(0)| < |f^{n - 1}(0)|$ for some $n \ge 3$. Then since $|f^{n-1}(0)| > 2$, 
\begin{align*}
\prod_{k=1}^{n - 2} |f^k(0)| & < (|f^{n - 1}(0)| - 1)^2 = |f^{n - 1}(0)|^2 - 2|f^{n - 1}(0)| + 1.
\end{align*}

We now show that $\prod_{k=1}^{n - 1} |f^k(0)| < |f^n(0)|$. For this, note that
\begin{align*}
\prod_{k=1}^{n - 1} f^k(0) & = |f^{n - 1}(0)| \cdot {\displaystyle \prod_{k=1}^{n - 2} |f^k(0)|}\\
					& < |f^{n - 1}(0)| \cdot (|f^{n - 1}(0)|^2 - 2|f^{n - 1}(0)| + 1) \\
                    & < |f^n(0)|,
\end{align*}
where the last inequality follows from equation \eqref{eq:PrimPrime}.

Setting $|f^n(0)| = P_n \cdot N_n$ in accordance with Lemma \ref{lem_Doerksen/Haensch}, we see that
\begin{align*}
N_n =  \prod_{d|n, d \not = n} P_d \leq \prod_{k = 1}^{n - 1} P_k \leq  \prod_{k = 1}^{n - 1} |f^k(0)| < |f^n(0)|.
\end{align*}
Hence we see that that $P_n > 1$, and thus $f^n(0)$ has a primitive prime divisor.
\end{proof}

We now prove that $D_{d,e,c}$ is finite if and only if $c = \pm 1$.

\begin{proof}[Proof of Theorem \ref{th:DFinite}]
In the forward direction we proceed by contradiction. Assume that $D_{d, e, c}$ is finite and $c \notin \{1, -1\}$. Let $M = \max D_{d, e, c}$. By Proposition \ref{primitive_prime}, we know that every term in $O_{d,e,c}$ has a primitive prime divisor. Suppose that $p$ is a primitive prime divisor of $f^M(0)$. Since $p \mid f^M(0)$, it follows that the period of $0$ modulo $p$ is $M$, and thus $M \le p$. If $M < p$, then $v_{p}(f^M(0)) > v_{p}(M)$, and hence $Mp \in D_{d,e,c}$ by Proposition \ref{pr:DivSetProp}.\eqref{d:valuation}.  This is a contradiction to the maximality of $M$. 

Now consider the case $p = M$. Since $p \mid f^p(0)$, write $f^p(0) = mp$ where $m \in \mathbb{Z}$. As a consequence of Lemma \ref{lem:increasing}, $m > 1$, so there is some prime $q$ such that $q \mid m$. This means that $pq \mid f^p(0)$. Since $O_{d,e,c}$ is a divisibility sequence, $p \mid pq$ implies $f^p(0) \mid f^{pq}(0)$. Therefore $pq \mid f^{pq}(0)$.  So $pq \in D_{d, e, c}$, which is a contradiction.

In the reverse direction we show that if $c \in \{1, -1\}$, then $D_{d, e, \pm 1}$ is finite. Our approach is to show that $D_{d,e,\pm 1}$ does not contain any primes.  By Proposition \eqref{d:smallest_prime}, this is sufficient to show that $D_{d,e,\pm 1} = \{1\}$ .

If $d$ or $e$ is even, then by Proposition \ref{only_primes} there are no primes in $D_{d, e, c}$ except the divisors of $c$. Since $c = \pm 1$, there are no primes in $D_{d, e, \pm 1}$.

When $d$ and $e$ are both odd and $c = 1$, then $f(-1) = -1$. Since $-1$ is a fixed point, $0$ can not have period $p$ modulo any prime $p$. Therefore $D_{d,e, 1}$ contains no primes.

When $d$ and $e$ are both odd and $c = -1$, a similar argument can be made. In this case $1$ is a fixed point, and once again $D_{d, e, c}$ contains no primes.
\end{proof}

\section{Restriction of primes in the divisibility set}

In this section, we provide conditions that would prevent primes from appearing in the index divisibility set of $f(x) = x^d + x^e + c$.  By Proposition \ref{only_primes}, we know that when $d$ and $e$ are both odd, the divisibility set $D_{d,e,c}$ may contain primes that do not divide $c$.  Indeed, we find examples of this: $31 \in D_{13,3,5}$, $157 \in D_{107,3,60}$, $223 \in D_{77,3,74}$, among many others.  

As stated several times previously, for a prime $p$ to be in the index divisibility set, either $p \mid c$ or $0$ has period $p$ modulo $p$.  In the latter case, the map $f(x)$ is a cyclic permutation of $\bbz/p\bbz$.  The conditions that restrict primes from appearing in a divisibility set result from showing that $f$ is not a cyclic permutation, either because it is not a permutation or because its permutation type is not a $p$-cycle.  All the computations in this section are local and thus apply to any map that is congruent to $f(x)$ modulo $p$. 

We also note that if $d \equiv e \pmod {p-1}$, then $x^d + x^e + c \equiv 2x^d + c \pmod p$.  We treat this as a separate case later in the section.

\subsection{The case \texorpdfstring{$d \not\equiv e \pmod{p-1}$}{}}
Let $D$ denote the index divisibility set for $f(x) \in \bbz[x]$, and let $\ord_p(a)$ denote the multiplicative order $a$ in $(\bbz/p\bbz)^\times$.

\begin{prop}
Suppose $f(x) \equiv x^d + x^e + c \pmod p$, where $0 < e < d < p$.  Then $p \notin D$ if any of the following is true:
\begin{enumerate}
\item $d$ or $e$ is even and $p \nmid c$;
\item $(p-1)/\gcd(d-e,p-1)$ is even;
\item $\ord_p(2) \nmid \gcd(d-e,p-1)$;
\item $\gcd(d-e,p-1) < \log_2(p)$.
\end{enumerate}
\end{prop}

\begin{proof}
The first statement is effectively a restatement of Proposition \ref{only_primes}.

For the next two cases, we recall that if $p \nmid c$, then $p \in D$ if and only if $0$ is $p$-periodic modulo $p$.  In particular, if $f(x)$ is not injective, then $p \notin D$.  By definition, $f(x)$ is injective if $f(x) - a$ has a root modulo $p$ for each $a \in \bbz/p\bbz$, and as $c$ is arbitrary, injectivity is equivalent to showing that $f(x)$ has a root modulo $p$ for all $c \in \bbz/p\bbz$.  That is, $f$ is injective if and only if $\res(f(x),x^{p-1}-1) \equiv c^{p-1}-1 \pmod p$, where
\begin{align} \label{eq:Res(f,p)}
\renewcommand{\arraystretch}{1.4}
\res(f(x),x^{p-1}-1) = \det
\bmat{
\begin{tikzpicture}[scale=.4]
\draw (0,0) node[fill=white] {$1$} -- ++(3,-3) node[fill=white] {$1$} ++(1,-1) node[fill=white] {$c$} -- ++(5,-5) node[fill=white] {$c$}
	(0,-6) node[fill=white] {$-1$} -- ++(3,-3) node[fill=white] {$-1$}
	(0,-6) node[fill=white] {$-1$} -- ++(3,-3) node[fill=white] {$-1$}
	(4,0) node[fill=white] {$1$} -- ++(5,-5) node[fill=white] {$1$}
	(4,-2.5) node[fill=white] {$1$} -- ++(5,-5) node[fill=white] {$1$};
\end{tikzpicture}
}
\end{align}
is the resultant of $f(x)$ and $x^{p-1}-1$.  This resultant is the determinant of a $(d+p-1) \times (d+p-1)$ matrix, where the entries in the first $d$ columns correspond to the coefficients of $x^{p-1}-1$, and the entries in the last $p-1$ columns correspond to the coefficients of $f(x)$.  For simplicity, only the nonzero entries are shown in \eqref{eq:Res(f,p)}.  The $-1$'s in the bottom left of the matrix may be eliminated using elementary row operations, reducing the computation to the determinant of the following $(p-1) \times (p-1)$ matrix:
\begin{align*}
\res(f(x),x^{p-1}-1) = \det \bmat{
\begin{tikzpicture}[scale=.6]
\draw (0,0) node[fill=white] {$c$} -- ++(5,-5) node[fill=white] {$c$}
	(1.5,0) node[fill=white] {$1$} -- ++(3.5,-3.5) node[fill=white] {$1$}
	(0,-4) node[fill=white] {$1$} -- ++(1,-1) node[fill=white] {$1$}
	(3.5,0) node[fill=white] {$1$} -- ++(1.5,-1.5) node[fill=white] {$1$}
	(0,-2) node[fill=white] {$1$} -- ++(3,-3) node[fill=white] {$1$};
\end{tikzpicture}
}.
\end{align*}
We note that the matrix itself is the circulant matrix for $f(x)$.  That is, the entries in the top row correspond to the coefficients of $f(x)$, and otherwise, the entries in each proceeding row are shifted by one to the right relative to the row above.  The determinant of this matrix is a polynomial in $c$.  Rather than compute all of the coefficients of this polynomial, we concentrate on the constant term as it is simpler to compute.  The constant term is
\begin{align} \label{eq:ResConst}
\det \bmat{
\begin{tikzpicture}[scale=.6]
\draw (1.5,0) node[fill=white] {$1$} -- ++(3.5,-3.5) node[fill=white] {$1$}
	(0,-4) node[fill=white] {$1$} -- ++(1,-1) node[fill=white] {$1$}
	(3.5,0) node[fill=white] {$1$} -- ++(1.5,-1.5) node[fill=white] {$1$}
	(0,-2) node[fill=white] {$1$} -- ++(3,-3) node[fill=white] {$1$};
\end{tikzpicture}
}.
\end{align}
In particular, if the determinant of this matrix in equation \eqref{eq:ResConst} is not congruent to $-1$ modulo $p$, then $\res(f(x),x^{p-1}-1) \not\equiv c^{p-1}-1 \pmod p$.

The matrix in equation \eqref{eq:ResConst} is the $(p-1) \times (p-1)$ circulant matrix for the polynomial $g(x) = x^d + x^e$, and the determinant of such a matrix is
\begin{align*}
\prod_{n=1}^{p-1} g(\zeta^n),
\end{align*}
where $\zeta$ is a primitive $(p-1)$-st root of unity.  Computing this product, we have
\begin{align*}
\prod_{n=1}^{p-1} g(\zeta^n) &= \prod_{n=1}^{p-1} \zeta^{dn} + \zeta^{en}\\
&= \prod_{n=1}^{p-1} \zeta^{en}\prod_{n=1}^{p-1} (\zeta^{(d-e)n} + 1)\\
&= \left(\prod_{n=1}^{p-1} \zeta^n\right)^e\left(\prod_{n=1}^{(p-1)/k} (\zeta^{(d-e)n} + 1)\right)^k,
\end{align*}
where $k = \gcd(d-e,p-1)$.  Note that the first product is the product of the roots of $x^{p-1}-1$, while the second is the product of the roots of $(x-1)^{(p-1)/k}-1$, hence
\begin{align*}
\prod_{n=1}^{p-1}\zeta^n = -1 \atext{and} \prod_{n=1}^{(p-1)/k} (\zeta^{(d-e)n}+1) = \begin{cases*}
0 & if $(p-1)/k$ is even \\
-2 & if $(p-1)/k$ is odd. 
\end{cases*}
\end{align*}
Since $e$ is odd, and $k$ is even when $(p-1)/k$ is odd, we have
\begin{align} \label{eq:MatDet}
\left(\prod_{n=1}^{p-1} \zeta^n\right)^e\left(\prod_{n=1}^{(p-1)/k} (\zeta^{(d-e)n} + 1)\right)^k = \begin{cases*}
0 & if $(p-1)/k$ is even \\
-2^k & if $(p-1)/k$ is odd.
\end{cases*}
\end{align}

Note that $-2^k \equiv -1 \pmod{p-1}$ if and only if $\ord_p(2) \mid k$.  Thus if $(p-1)/k$ is even or $\ord_p(2) \nmid k$, then $\res(f(x),x^{p-1}-1) \not\equiv c^{p-1}-1 \pmod p$.  Therefore $f(x)$ is not injective and $p \notin D$.

Finally, we note that $\log_2(p) < \ord_p(2) \le p-1$ and $2 \le k < p-1$. Hence if $k < \log_2(p)$, then $\ord_p(2) \nmid k$, and so $p \notin D$.  While this statement only takes advantage of the trivial bounds for $\ord_2(p)$ and $k$, it not require the exact value of $\ord_2(p)$.  
\end{proof}

We note that the determinant in equation \eqref{eq:ResConst} may also be computed directly; we outline one alternative proof.  First, shift the each of the first $e$ columns by $p-e-1$ to the right so that 1's are on the diagonal.  This shifting requires an odd number of column swaps, so 
\begin{align*}
\det \bmat{
\begin{tikzpicture}[scale=.6]
\draw (1.5,0) node[fill=white] {$1$} -- ++(3.5,-3.5) node[fill=white] {$1$}
	(0,-4) node[fill=white] {$1$} -- ++(1,-1) node[fill=white] {$1$}
	(3.5,0) node[fill=white] {$1$} -- ++(1.5,-1.5) node[fill=white] {$1$}
	(0,-2) node[fill=white] {$1$} -- ++(3,-3) node[fill=white] {$1$};
\end{tikzpicture}
} 
= -\det \bmat{
\begin{tikzpicture}[scale=.6]
\draw (0,0) node[fill=white] {$1$} -- ++(5,-5) node[fill=white] {$1$}
	(1.5,0) node[fill=white] {$1$} -- ++(3.5,-3.5) node[fill=white] {$1$}
	(0,-4) node[fill=white] {$1$} -- ++(1,-1) node[fill=white] {$1$};
\end{tikzpicture}
}.
\end{align*}
Then, by exploiting the fact that there are exactly two 1's in each row and each column, the matrix may be arranged, via an even number of swaps, into block diagonal form:
\begin{align*}
-\det \bmat{
\begin{tikzpicture}[scale=.6]
\draw (0,0) node[fill=white] {$1$} -- ++(5,-5) node[fill=white] {$1$}
	(1.5,0) node[fill=white] {$1$} -- ++(3.5,-3.5) node[fill=white] {$1$}
	(0,-4) node[fill=white] {$1$} -- ++(1,-1) node[fill=white] {$1$};
\end{tikzpicture} 
}
= -\det \bmat{B_1 \\ & B_2 \\ & & \ddots \\ & & & B_k}
\end{align*}
where $k = \gcd(d-e,p-1)$, and each block is a $((p-1)/k) \times ((p-1)/k)$ matrix with $1$'s on the diagonal, superdiagonal, and in the bottom left corner:
\begin{align*}
B_1 = \cdots = B_k = \bmat{1 & 1 \\ & 1 & 1 \\ & & \ddots & \ddots \\  & & & 1 & 1\\ 1 & & & & 1}.
\end{align*}
Since
\begin{align*}
\det B_i &= \det \bmat{1 \\ & 1 \\ & & \ddots \\ & & & 1} + \det \bmat{0 & 1 \\ & \ddots & \ddots \\ & & 0 & 1 \\ 1 & & & 0} \\
&= \begin{cases*}
0 & if $(p-1)/k$ is even \\
2 & if $(p-1)/k$ is odd,
\end{cases*}
\end{align*}
it follows that 
\begin{align*}
-\det \bmat{B_1 \\ & B_2 \\ & & \ddots \\ & & & B_k} = \begin{cases*}
0 & if $(p-1)/k$ is even \\
-2^k & if $(p-1)/k$ is odd,
\end{cases*}
\end{align*}
which agrees with equation \eqref{eq:MatDet}.

\subsection{The case \texorpdfstring{$d \equiv e \pmod{p-1}$}{}}

In the case that $d \equiv e \pmod{p-1}$, we have $x^d + x^e + c \equiv 2x^d + c \pmod p$.  We obtain a very simple condition in the case $d \equiv 1 \pmod{p-1}$.

\begin{prop}
If $f(x) \in \bbz[x]$ and $f(x) \equiv ax+c \pmod p$, then $p \in D$ only if $a \equiv 1 \pmod p$ or $c \equiv 0 \pmod p$.
\end{prop}

\begin{proof}
A simple induction shows that 
\begin{align*}
f^p(x) = a^px + c\left(\sum_{i=0}^{p-1} a^i\right) \equiv \begin{cases*}
ax & if $a \equiv 1 \pmod p$\\
ax + c & if $a \not\equiv 1 \pmod p$.
\end{cases*}
\end{align*}
The result follows immediately.
\end{proof}

Returning to the map $2x^d + c$, we note that $\tau(x) = x + 1$ and $\sigma(x) = 2x$ are permutations of $\bbz/p\bbz$, and the map $x^d$ is a permutation of $\bbz/p\bbz$ if and only if $\gcd(d,p-1) = 1$.  Therefore $f(x) = \tau^c \circ \sigma \circ \pi (x)$ is a permutation of $\bbz/p\bbz$ if and only if $\gcd(d,p-1) = 1$.  

Moreover, cyclic permutations of $\bbz/p\bbz$ are even.  Hence if $f(x)$ is an odd permutation of $\bbz/p\bbz$, then $p$ is not in the divisibility set of $f(x)$.

\begin{lem} \label{lem:PowerMapOdd}
If $p \equiv 1 \pmod 4$, then $x^d$ is a odd permutation of $\bbz/p\bbz$ if and only if $d \equiv 3 \pmod 4$.
\end{lem}

\begin{proof} 
See the proof of \cite[Theorem 1.3]{cgs17}.
\end{proof}

\begin{prop} \label{pr:restrict}
Suppose $f(x) \in \bbz[x]$ and $f(x) \equiv ax^d + c \pmod p$, where $p \equiv 1 \pmod 4$, $d \equiv 3 \pmod 4$, and $\ord_p(a)$ is odd.  Then $p \notin D$.
\end{prop}

\begin{proof}
The translation map $\tau(x) = x + 1$ is a cyclic permutation of $\bbz/p\bbz$ and is even.  Since $\ord_p(a)$ is odd, the cycle $(a, a^2, a^3, \ldots, a^{\ord_p(a)})$ is an even permutation, hence the scaling map $\sigma(x) = ax$ is an even permutation. Finally, $\pi(x) = x^d$ is an odd permutation by Lemma \ref{lem:PowerMapOdd}.  Thus $f(x)$ is an odd permutation of $\bbz/p\bbz$.
\end{proof}

For our polynomial $2x^d+c$, the conditions $p \equiv 1 \pmod 4$ and $\ord_p(2)$ is odd in Proposition \ref{pr:restrict}, when taken together, are equivalent to $p \equiv 1 \pmod 8$.  The reason for this is that $2$ is not a quadratic residue if $p \equiv 5 \pmod 8$, and therefore the order of $2$ is even.  In particular, in order for $\ord_p(2)$ to be odd, it must be that $2$ is a $2^v$-th power in $\bbz/p\bbz$, where $v = v_2(p-1)$.  There are $(p-1)/2^v$ values which are $2^v$-th powers modulo $p$, so if $p\equiv 1 \pmod 8$ and we assume the heuristic that all values are equally likely to generate $(\bbz/p\bbz)^\times$ (c.f. Artin's conjecture), then the probability that that $2$ is a $2^v$-th power given that it is already a square is
\begin{align*}
\frac{(p-1)/2^v}{1/2} = \frac{1}{2^{v-1}}.
\end{align*}
The primes that are congruent to $1$ modulo $8$ may be partitioned into sets of the form $p \equiv 2^{k-1}+1 \pmod{2^k}$ for $k \ge 4$.  As primes are distributed equally across equivalence classes, the proportion of primes satisfying $p \equiv 2^{k-1}+1 \pmod{2^k}$ is $1/2^{k-1}$.  Thus we expect that the proportion of all primes where $p \equiv 1 \pmod 8$ and $\ord_p(2)$ is
\begin{align*}
\sum_{k=4}^\infty \frac{1}{2^{k-1}} \cdot \frac{1}{2^{k-2}} = \frac{1}{24},
\end{align*}
and therefore Proposition \ref{pr:restrict} is only sufficient to remove $1/24$ of all primes from consideration.  

Given a sequence $(a_n)$, the rank of apparition function $t(x)$ gives the minimum value $n$ such that $x \mid a_n$.  This function plays a key role in the study of Lucas sequences \cite{k17} and elliptic divisibility sequences \cite{s17}.  In our case, the rank of apparition is the period of $0$ modulo $x$.  It would be interesting to see if the methods of Sanna and Kim can be translated to the dynamical setting to give more concrete results regarding primes in index divisibility sets.
\bibliographystyle{abbrv}
\bibliography{DDS}

\end{document}